\documentclass[12pt, reqno]{amsart}

\usepackage{latexsym}
\usepackage{amssymb}
\usepackage{mathrsfs}
\usepackage{amsmath}
\usepackage{fancybox,color}
\usepackage{enumerate}
\usepackage[latin1]{inputenc}
\usepackage{eurosym}

\newcommand{\kom}[1]{}
\renewcommand{\kom}[1]{{\bf [#1]}}

 \def\1{\raisebox{2pt}{\rm{$\chi$}}}

\newtheorem{theorem}{Theorem}[section]
\newtheorem{corollary}[theorem]{Corollary}
\newtheorem{lemma}[theorem]{Lemma}
\newtheorem{proposition}[theorem]{Proposition}
\newtheorem{remark}[theorem]{Remark}

\newcommand{\RR}{{\mathbb R}}

\newcommand{\N}{{\mathbb N}}

 \def\1{\raisebox{2pt}{\rm{$\chi$}}}
 
\newcommand{\norm}[1]{\left|\left|#1\right|\right|}

\def\vint_#1{\mathchoice%
          {\mathop{\kern 0.2em\vrule width 0.6em height 0.69678ex depth -0.58065ex
                  \kern -0.8em \intop}\nolimits_{\kern -0.4em#1}}%
          {\mathop{\kern 0.1em\vrule width 0.5em height 0.69678ex depth -0.60387ex
                  \kern -0.6em \intop}\nolimits_{#1}}%
          {\mathop{\kern 0.1em\vrule width 0.5em height 0.69678ex
              depth -0.60387ex
                  \kern -0.6em \intop}\nolimits_{#1}}%
          {\mathop{\kern 0.1em\vrule width 0.5em height 0.69678ex depth -0.60387ex
                  \kern -0.6em \intop}\nolimits_{#1}}}
\def\vintslides_#1{\mathchoice%
          {\mathop{\kern 0.1em\vrule width 0.5em height 0.697ex depth -0.581ex
                  \kern -0.6em \intop}\nolimits_{\kern -0.4em#1}}%
          {\mathop{\kern 0.1em\vrule width 0.3em height 0.697ex depth -0.604ex
                  \kern -0.4em \intop}\nolimits_{#1}}%
          {\mathop{\kern 0.1em\vrule width 0.3em height 0.697ex depth -0.604ex
                  \kern -0.4em \intop}\nolimits_{#1}}%
          {\mathop{\kern 0.1em\vrule width 0.3em height 0.697ex depth -0.604ex
                  \kern -0.4em \intop}\nolimits_{#1}}}

\newcommand{\intav}{\vint}
\newcommand{\aveint}[2]{\mathchoice%
          {\mathop{\kern 0.2em\vrule width 0.6em height 0.69678ex depth -0.58065ex
                  \kern -0.8em \intop}\nolimits_{\kern -0.45em#1}^{#2}}%
          {\mathop{\kern 0.1em\vrule width 0.5em height 0.69678ex depth -0.60387ex
                  \kern -0.6em \intop}\nolimits_{#1}^{#2}}%
          {\mathop{\kern 0.1em\vrule width 0.5em height 0.69678ex depth -0.60387ex
                  \kern -0.6em \intop}\nolimits_{#1}^{#2}}%
          {\mathop{\kern 0.1em\vrule width 0.5em height 0.69678ex depth -0.60387ex
                  \kern -0.6em \intop}\nolimits_{#1}^{#2}}}

\newcommand{\re}{\mathbb{R}}
\newcommand{\rn}{\mathbb{R}^n}

\begin{document}
\title[
the fractional maximal function of a radial function
]
{The variation of the fractional maximal function of a radial function
}
\author{Hannes Luiro and Jos{\'e} Madrid}

\date{\today}

\address{Department of Mathematics and Statistics, University of Jyvaskyla,P.O.Box 35 (MaD), 40014 University of Jyvaskyla, Finland}
\email{hannes.s.luiro@jyu.fi}
\address{Department of Mathematics, Aalto University, P.O. Box 11100, FI--00076 Aalto University, Finland}
\email{jose.madridpadilla@aalto.fi}
\address{The Abdus Salam International Centre for Theoretical Physics, Str. Costiera 11, 34151 Trieste, Italy}
\email{jmadrid@ictp.it}

\maketitle
\begin{abstract}In this paper we study the regularity of the non-centered fractional maximal operator $M_{\beta}$. As the main result, we prove that   
there exists $C(n,\beta)$ such that if $q=n/(n-\beta)$ and $f$ is \textit{radial function}, then 
$\|DM_{\beta}f\|_{L^{q}(\rn)}\leq C(n,\beta)\|Df\|_{L^{1}(\rn)}$. 
The corresponding result was previously known only if $n=1$ or $\beta=0$. Our proofs are almost free from one-dimensional arguments. Therefore, we believe that the new approach may be very useful when trying to extend the result for all $f\in W^{1,1}(\rn)$.  
\end{abstract}

\section{Introduction}
The non-centered fractional Hardy-Littlewood maximal operator $M_{\beta}$ is defined by
setting for $f\in L^1_{loc}(\rn)\,$ and $0\leq \beta<n\,$ that
\begin{equation}\label{eq:1}
M_{\beta} f(x):=\sup_{B(z,r)\ni x}
\frac{r^\beta}{|B(z,r)|}\int_{B(z,r)}|f(y)|\,dy\,=:\,\sup_{B(z,r)\ni x}r^{\beta}\intav_{B(z,r)}|f(y)|\,dy\,
\end{equation}
for every $x\in\rn\,$. The centered version of $M_{\beta}$, denoted by $M^{c}_\beta$, is defined by taking the supremum over all balls centered at $x$. In the non-fractional case $\beta=0$, we also denote $M_0=M$.


The study of the regularity of maximal operators has strongly attracted the attention of many authors in recent years. The boundedness of the classical maximal operator on the Sobolev space $W^{1,p}(\rn)$ for $p>1$ was established by Kinnunen in \cite{Ki}. The analogous result in the fractional context was established by Kinnunen and Saksman in \cite{KiSa}: for every $0<\beta<n$ we have that $M_{\beta}$ is bounded from $W^{1,p}(\rn)$ to $W^{1,q}(\rn)$ under the relation $1/q=1/p-\beta/n$ (if $p>1$).
For  other interesting results on this theory we refer to \cite{BCHP}, \cite{CFS}, \cite{CaHu}, \cite{CMP}, \cite{CaMo}, \cite{CaSv}, \cite{HM}, \cite{HO}, \cite{L}, \cite{Ma} and \cite{R}.

The case $p=1$ is particularly complicated and interesting. In the case $n=1$
it is known  (see \cite{Ta} and \cite{AlPe} for the Non-Centered, and \cite{Ku} for the Centered) that $Mf$ is weakly differentiable and
\begin{equation}\label{n=1 clasico}
\|DMf\|_{L^{1}(\mathbb{R})}\leq C\|Df\|_{L^{1}(\mathbb{R})},
\end{equation}
 but even in this case there are still some interesting open questions. The proofs of these theorems strongly exploit the simplicity of one-dimensional topology. Indeed, the situation in higher dimension is quite unknown, only a few results have been obtained (see \cite{L2}, \cite{S}).
\noindent
The analogous result to \eqref{n=1 clasico} for the fractional non-centered maximal operator was established by Carneiro and Madrid in \cite{CaMa}. In full generality the next question was posed by them.

\noindent {\bf Main Question.} Let $0 \leq \beta <n$ and $q = n/(n-\beta)$. Is the operator $f \mapsto |D M_{\beta}f|$ bounded from $W^{1,1}(\rn)$ to $L^q(\rn)$? 
\smallskip

The problem can be rather easily reduced to the case $0 \leq \beta <1$, as it was also observed by Carneiro and Madrid (see \cite{CaMa}). Indeed, in the case
$1 \leq \beta < d$, the positive answer follows by combining the boundedness property of the fractional maximal operator from $L^{p}$ to $L^{q}$ (under the condition $1/q=1/p-\beta/n$), the Sobolev embedding Theorem and the result  in \cite{KiSa} which says:
 If $f \in L^{r}(\rn)$ with $1 < r < n$ and $1 \leq \beta < n/r$, then $M_\beta f$ is weakly differentiable and 
$$\left| D M_{\beta} f(x)\right| \leq C(n,\beta) \,M_{\beta -1}f(x)\,\text{ for a.e.  }x\in\rn\,.$$


In the case $\beta=0$ (non-fractional operator) the main question for radial functions was recently proven by Luiro  \cite{L2}. Our main Theorem is a counterpart of this result in the case $\beta>0$.
\begin{theorem}[Main Theorem]\label{Main Theorem}
Given $0<\beta<n$ and $q=n/(n-\beta)$, there is a constant $C=C(n,\beta)>0$ such that for every radial function $f\in W^{1,1}(\rn)$ we have that $M_{\beta}f$ is weakly differentiable and
$$
\|DM_{\beta}f\|_{L^{q}(\rn)}\leq C\|Df\|_{L^{1}(\rn)}.
$$
\end{theorem}
The proof adapts some basic ideas from \cite{L2}, like in Lemma \ref{basic}. However, as we will see (and as one can see in \cite{CaMa} as well), some new difficulties arise
with respect to the case of the classical maximal operator. 
The key element to overcome these problems is Lemma \ref{lemma2}. We believe that the modification of this result may play a crucial role in the solution of the problem in its full generality.
In addition, we point out that the presented argument also gives a new proof for the case $n=1$\footnote{With slight modifications in the proof in the case $f$ is not symmetric with respect to the origin}, in other words our argument also gives a new proof for Theorem 1 in \cite{CaMa}.


\section{Preliminaries}\label{sec2}
Let us introduce some notation. The boundary of the $n$-dimensional unit ball is denoted by $S^{n-1}$. The $s$-dimensional Hausdorff measure is denoted by 
$\mathcal{H}^s$. The volume of the $n$-dimensional unit ball is denoted by $\omega_n$ and the $\mathcal{H}^{n-1}$-measure of $S^{n-1}$ by $\sigma_n$. The integral average of $f\in L^{1}_{loc}(\rn)$ over a measurable set $A\subset\rn$ is sometimes denoted by $f_A$.  The weak derivative of 
$f$ (if exists) is denoted by $Df$. 
If $v\in S^{n-1}$, then 
\begin{equation*}
D_vf(x):=\lim_{h\to 0}\frac{1}{h}(f(x+hv)-f(x))\,,
\end{equation*}
in the case the limit exists. For $f\in W^{1,1}(\rn)$, $0\leq \beta<n$, let us define
 \begin{align*}
     & \mathcal{B}^{\beta}_{x}(f)\\
=\,&\mathcal{B}_x:=\bigg{\{}B(z,r):x\in\bar{B}(z,r)\,,\,M_{\beta}f(x)=r^{\beta}\intav_{B(z,r)}|f(y)|\,dy\,\bigg{\}}.
 \end{align*}
We use to call $\mathcal{B}_x$ as the collection of the best balls at $x$. It is easy to see that $\mathcal{B}_x$ is non-empty set for every $x\in \RR^{n}$ (since $f\in L^{1}(\RR^{n})$) and also it is compact in the sense that if $B(z_k,r_k)\in\mathcal{B}_x$ and $z_k\to z\in\rn$ and $r_k\to r\in(0,\infty)$ as $k\to\infty$, then $B(z,r)\in\mathcal{B}_x$.

\begin{proposition}\label{afin map 1}
Given $f\in W^{1,1}(\rn)$, $B$ a ball, a family of affine mappings $L_{i}(y)=a_{i}y+b_{i}$, $a_{i}\in\mathbb{R}$, $b_{i}\in\rn$, $B_i:=L_i(B)$ and a sequence $\{h_{i}\}_{i\in\N}\subset\mathbb{R}$ such that $h_{i}\to0$ as $i\to\infty$, 
$$
\lim_{i\to\infty}\frac{L_{i}(y)-y}{h_{i}}=g(y)
\,\,\text{ and }\,\,
\lim_{i\to\infty}\frac{a_{i}^{\beta}-1}{h_{i}}=\gamma,
$$
where $\gamma\in\re, g:\rn\to\re\,$,  it holds that
 \begin{align*}
&\lim_{i\to\infty}\frac{1}{h_{i}}\left(r^{\beta}_{i}\intav_{B_{i}}|f(y)|dy-r^{\beta}\intav_{B}|f(y)|dy\right)\\
=\,&r^{\beta}\intav_{B}D|f|(y)\cdot g(y)dy+\gamma r^{\beta}\intav_{B}|f(y)|\,dy\,,
\end{align*}
where $r$ denotes the radius of $B$ and $r_{i}$ the radius of $B_{i}$ for every $i$.
\end{proposition}

\begin{proof}[Proof of Proposition \ref{afin map 1}]
\begin{eqnarray*}
&&\frac{1}{h_{i}}\left(r^{\beta}_{i}\intav_{B_{i}}|f(y)|dy-r^{\beta}\intav_{B}|f(y)|dy\right)\\
&&=\frac{1}{h_{i}}\left(a^{\beta}_{i}r^{\beta}\intav_{L_{i}(B)}|f(y)|dy-r^{\beta}\intav_{B}|f(y)|dy \right)\\
&&=r^{\beta}\left(\intav_{B}\frac{a^{\beta}_{i}|f(y+(L_{i}(y)-y))|-|f(y)|}{h_{i}}\right)\\
&&=r^{\beta}\left(\intav_{B}\frac{a^{\beta}_{i}|f(y+(L_{i}(y)-y))|-a^{\beta}_{i}|f(y)|}{h_{i}}dy\right.\\
&&\ \ \ \left.+\intav_{B}\frac{|f(y)|(a^{\beta}_{i}-1)}{h_{i}}dy\right)\\
&&\to r^{\beta}\intav_{B}D|f|(y)\cdot g(y)dy+\gamma r^{\beta}\intav_{B}|f(y)|dy\\
&& \text{as}\ \ i\to\infty.
\end{eqnarray*}
\end{proof}

\begin{lemma}\label{afin map 2}
Let $f\in W^{1,1}(\rn)$, $x\in\rn$, $B\in \mathcal{B}_{x}$, $\delta>0$, and let $L_{h}(y)=a_{h}y+b_{h}$, $h\in[-\delta,\delta]$, be affine mappings such that $x\in L_{h}(\overline B)$ 
and
$$
\lim_{h\to0}\frac{L_{h}(y)-y}{h}=g(y)
\,\,\text{ and }\,\,
\lim_{h\to 0}\frac{a_{h}^{\beta}-1}{h}=\gamma.
$$
Then
\begin{equation}\label{eq99}
0=(rad(B))^{\beta}\intav_{B}D|f|(y)\cdot g(y)dy+\gamma M_{\beta}f(x).
\end{equation}
\end{lemma}
\begin{proof}[Proof of Lemma \ref{afin map 2}]
By the previous Proposition \ref{afin map 1} the right hand side of (\ref{eq99}) equals to 
\begin{equation}
\lim_{h\to 0}\frac{1}{h}\left(rad(L_h(B))^{\beta}\intav_{L_h(B)}|f(y)|dy-r^{\beta}\intav_{B}|f(y)|dy\right)\,=:\lim_{h\to 0}\frac{1}{h}s_h\,.
\end{equation}
Since $B\in\mathcal{B}_x$ and $x\in L_h(\bar{B})$ for all $h$, it follows that $s_h\leq 0$ for all $h$. Since $h$ can take positive and negative values, the existing limit must equal to zero.
\end{proof}

\begin{corollary}\label{coro good eq}
In particular, if $L_{h}(y)=y+h(y-x)$ we obtain that $g(y)=y-x$, $\gamma=\beta$, and $x=L_h(x)\in L_h(\bar{B})$. Therefore
\begin{equation}\label{nice identity}
(rad(B))^{\beta}\intav_{B}D|f|(y)\cdot x dy=(rad(B))^{\beta}\intav_{B}D|f|(y)\cdot y\,dy+\beta M_{\beta}f(x).
\end{equation}
\end{corollary}

The following lemma is a counterpart of Lemma 2.2 in \cite{L2}. It was proved by Carneiro and Madrid in \cite[Theorem 1]{CaMa} that if $g\in W^{1,1}(\RR)$ thus $M_{\beta}{g}$ is absolutely continuous, therefore if $f\in W^{1,1}(\RR^{n})$ is a radial functions we can apply the next lemma.
\begin{lemma}\label{basic}
Suppose that $f\in W^{1,1}(\rn)$ and $M_{\beta}f$ is differentiable at $x$. Then
\begin{itemize}
\item[(1)]
For all $v\in S^{n-1}$ and $B\in\mathcal{B}_x\,$, it holds that 
\begin{equation*}
D_{v}M_{\beta}f(x)=(rad(B))^{\beta}\intav_{B}D_{v}|f|(y)\,dy\,.
\end{equation*}
\item[(2)]
If $x\in B$ for some $B\in \mathcal{B}_x$, then $DM_{\beta}f(x)=0\,.$
\item[(3)]
If $x\in \partial B$, $B=B(z,r)\in \mathcal{B}_x$ and $DM_{\beta}f(x)\not=0$, then
\begin{equation*}
\frac{DM_{\beta}f(x)}{|DM_{\beta}f(x)|}=\frac{z-x}{|z-x|}\,.
\end{equation*}
\item[(4)] If $B\in\mathcal{B}_x$, then 
\begin{equation}\label{char}
\int_{B}|f|(y)dy\,=\,-\frac{1}{\beta}\int_{B}D|f|(y)\cdot (y-x)\,dy\,.
\end{equation}
\end{itemize}
\end{lemma}

\begin{proof}
\begin{itemize}
\item[(1)] Let $B=B(z,r)\in\mathcal{B}_x\,$ and  $B_{h}:=B(z+hv,r)$. Then it holds for every $v\in S^{n-1}$ that 
\begin{eqnarray*}
\lim_{h\to 0} \frac{M_{\beta}f(x+hv)-M_{\beta}f(x)}{h}
&\geq& \lim_{h\to 0}\frac{r^{\beta}\intav_{B_{h}}|f(y)|dy-r^{\beta}\intav_{B}|f(y)|dy}{h}\\
&=&r^{\beta}\intav_{B}D_{v}|f|(y)dy\\
&=& \lim_{h\to 0}\frac{r^{\beta}\intav_{B}|f(y)|dy-r^{\beta}\intav_{B_{-h}}|f(y)|dy}{h}\\
&\geq&  \lim_{h\to 0} \frac{M_{\beta}f(x)-M_{\beta}f(x-hv)}{h}\,.
\end{eqnarray*}
\item[(2)] If $B\in \mathcal{B}_x$ and $x\in B$, then $M_{\beta}f(x)\leq M_{\beta}f(y)$ for every $y\in B$. 
\item[(3)] 

  Let  $B=B(z,r)\in\mathcal{B}_x$, $v\in S^{n-1}$ such that $v\cdot(z-x)=0$, and let us denote for all $h\in(0,\infty)$ that $x_h:=x+hv$, 
$r_h:= |z-x_h|$, and $B_h:=B(z,r_h)$. These definitions guarantee that $x_h\in \bar{B}_h\setminus B$ for all $h$, and $B\subset B_h$. Moreover, 
since $v\cdot(z-x)=0$, it is elementary fact that
\begin{equation*}
r_h=|z-x-hv|\leq |z-x|+\frac{h^2}{2r}\,.
\end{equation*}
 Therefore, $r/r_h\geq 1-(\frac{h}{r})^2\,$, and
\begin{align*}
M_{\beta}f(x_h)&\geq r_h^{\beta}\intav_{B_h}|f(z)|\,dz\,\geq\, \frac{r_h^{\beta}|B|}{r^{\beta}|B_h|}r^{\beta}\intav_{B}|f(z)|\,dz\\
=&\, \bigg(\frac{r}{r_h}\bigg)^{n-\beta}r^{\beta}\intav_{B}|f(z)|\,dz
\geq \,\bigg(1-\frac{h^2}{r^2}\bigg)^{n-\beta} M_{\beta}f(x)\,.
\end{align*}  
This implies  that $D_{v}M_{\beta}f(x)\geq 0\,$ for all $v\in S^{n-1}$ such that $v\cdot(z-x)=0\,$. Since we assumed that $M_{\beta}f$ is differentiable at $x$, it follows that 
\begin{equation*}
D_{v}M_{\beta}f(x)=0\, \text{ if }v\in S^{n-1}, v\cdot(z-x)=0\,.
\end{equation*}
 In particular, it follows that $DM_{\beta}f(x)$ is parallel to $z-x$ or $x-z$. The final claim follows easily 
by the fact that $M_{\beta}f(x+h(z-x))\geq M_{\beta}f(x)$ if $0\leq h\leq 2$.


\item[(4)] This is an immediate consequence of Corollary \ref{coro good eq}.
\end{itemize}

\end{proof}

\begin{proposition}\label{prop3}
 If $f\in W^{1,1}_{loc}(\rn)$, $z\in\rn$, $r>0$, then 
 \begin{equation}
  \intav_{B(z,r)}D|f|(y)\cdot(z-y)\,dy\,=\,n\bigg[\intav_{B(z,r)}|f|-\,\intav_{\partial B(z,r)}|f|\bigg]\,.
 \end{equation}
\end{proposition}
\begin{proof}
In the case of radial functions the previous proposition follows from the one dimensional case (that is enough in order to get Theorem \ref{Main Theorem}). In general, the proof of this proposition is based in the following fact, which is a consequence of Gauss Divergence Theorem.

{\remark[Integration by parts]{
 Given $\Omega\subset \rn$ a bounded open set with $C^{1}$ boundary and $\nu$ denotes the outward unit normal to $\partial \Omega$ if $u\in W^{1,p}(\Omega)$ and $v\in W^{1,q}(\Omega)$ for exponents p,q with
$$
\frac{1}{p}+\frac{1}{q}\leq 1+\frac{1}{n}
$$
thus the following identity holds
$$
\int_{\Omega}u\frac{\partial v}{\partial y_{i}}=\int_{\partial\Omega}uv\nu_{i}-\int_{\Omega}v\frac{\partial u}{\partial y_{i}}
$$
where $\nu_{i}$ is the $i-$component of the vector $\nu$.
}}

Using this we get
\begin{eqnarray*}
&&\int_{B(z,r)}D|f|(y)\cdot(z-y)dy\\
&&=\sum_{i=1}^{n}\int_{B(z,r)}\frac{\partial|f|}{\partial y_{i}}(y)(z_{i}-y_{i})dy\\
&&=\sum_{i=1}^{n}\left(\int_{\partial B(z,r)}|f(y)|(z_{i}-y_{i})\cdot\frac{(y_{i}-z_{i})}{|y-z|}dy+\int_{B(z,r)}|f(y)|dy\right)\\
&&=-\int_{\partial B(z,r)}|f(y)|\sum_{i=1}^{n}\frac{|y_{i}-z_{i}|^{2}}{|y-z|}dy+\sum_{i=1}^{n}\int_{B(z,r)}|f(y)|dy\\
&&=n\int_{B(z,r)}|f(y)|dy-\int_{\partial B(z,r)}|f(y)||y-z|dy\\
&&=n\left[\int_{B(z,r)}|f(y)|dy-\frac{r}{n}\int_{\partial B(z,r)}|f(y)|dy\right]\\
&&=n\left[\int_{B(z,r)}|f(y)|dy-\frac{r^{n}w_{n}}{r^{n-1}\sigma_{n}}\int_{\partial B(z,r)}|f(y)|dy\right]\\
&&=n\left[\int_{B(z,r)}|f(y)|dy-|B(z,r)|\intav_{\partial B(z,r)}|f(y)|dy\right].
\end{eqnarray*}

\noindent
By dividing both sides of the last equality by $|B(z,r)|$ we arrived in the desired identity.

\end{proof}

By using Proposition \ref{prop3} we yet  state one more formula related to the derivative of the  fractional maximal operator.
\begin{lemma}\label{lemma1}
Suppose that $f\in W^{1,1}_{loc}(\rn)$, $0<\beta<n$, $B\in \mathcal{B}^{\beta}_x$ for some $x\in\rn$, and $r:=rad(B)\,$. Then
\begin{equation}
\bigg|\intav_{B}D|f|(y)\,dy\,\bigg|\,=\,\frac{n}{r}\bigg[\big(1-\beta/n)\intav_{B}|f(y)|\,dy\,-\intav_{\partial B}|f(y)|dy\,\bigg]\,.
\end{equation}
\end{lemma}
\begin{proof}
Suppose that $B=B(z,r)$. By Lemma \ref{basic} and Proposition \ref{prop3} it follows that
\begin{align*}
&\bigg|\intav_{B}D|f|(y)\,dy\,\bigg|\,=\,\intav_{B}D|f|(y)\cdot\bigg(\frac{z-x}{r}\bigg)\,dy\,\\
=&\frac{1}{r}\bigg[\intav_{B}D|f|(y)\cdot(z-y)\,dy\,+\intav_{B}D|f|(y)\cdot(y-x)\,dy\,\bigg]\,\\
=&\frac{1}{r}\bigg[\intav_{B}D|f|(y)\cdot(z-y)\,dy\,-\beta\intav_{B}|f|(y)\,dy\,\bigg]\\
=&\frac{1}{r}\bigg[n\bigg[\intav_{B}|f|-\,\intav_{\partial B}|f|\bigg]-\beta\intav_{B}|f|(y)\,dy\,\bigg]\\
=&\frac{n}{r}\bigg[\big(1-\beta/n)\intav_{B}|f(y)|\,dy\,-\intav_{\partial B}|f(y)|\,\bigg]\,.
\end{align*}
\end{proof}

We will use the following elementary property for radial functions. The proof is left for an interested reader.
\begin{proposition}\label{radial1}
Suppose that  $f\in L^{1}_{loc}(\rn)$ satisfies $f(x)=F(|x|)$, $F:(0,\infty)\to[0,\infty)$, $B:=B(z,r)\subset B(0,2|z|)\setminus B(0,\frac{1}{2}|z|)\,$, and 
$a:= |z|-r$, $b:=|z|+r$. 
Then 
it holds that
\begin{equation}
\intav_{[a,b]}F(t)\,dt\,\leq C(n)\intav_{B(z,2r)}f(y)\,dy\,.
\end{equation}
\end{proposition}

The following two lemmas contain the key estimates for the proof of the main theorem.
\begin{lemma}\label{lemma3}
Suppose that  $f\in W^{1,1}_{loc}(\rn)$ is radial and $B\in \mathcal{B}^{\beta}_x$ for some $x\in\rn\setminus\{0\}$ such that $B\subset B(0,|x|)$.  Then
\begin{equation}
\bigg|\intav_{B}D|f|(y)\,dy\,\bigg|\,\leq\,\intav_{B}|Df(y)|\frac{|y|}{|x|}dy\,.
\end{equation}
\begin{proof}
If $|DM_\beta f(x)|=0$, the claim is trivial. If $|DM_\beta f(x)|\not=0$, Lemma \ref{basic} yields that
\begin{equation}
\frac{\int_{B}D|f|(y)\,dy\,}{\bigg|\int_{B}D|f|(y)\,dy\,\bigg|}\,=\,\frac{-x}{|x|}\,,
\end{equation}
and 
\begin{align*}
&\bigg|\int_{B}Df(y)\,dy\,\bigg|\,=\,\int_{B}D|f|(y)\cdot \frac{-x}{|x|}\,dy\,\\
=&\,\int_{B}Df(y)\cdot \frac{-y}{|x|}\,dy\,-\frac{\beta}{|x|}\int_{B}|f(y)|\,dy\,
\leq\int_{B}|Df(y)|\frac{|y|}{|x|}\,.
\end{align*}
This proves the claim.
\end{proof}
\end{lemma}

Given a ball $B=B(z,r)$ we define $2B$ to be equal to $B(z,2r)$.
\begin{lemma}\label{lemma2}
Suppose that $f\in W^{1,1}_{loc}(\rn)$ is radial, $0<\beta<n$, $B\in \mathcal{B}^{\beta}_x$ for some $x\in\rn$, $r:=rad(B)\leq\frac{|x|}{4}\,$, and
\begin{equation}
E:=\{z \in 2B\,:\,\frac{1}{2}|f|_{B}\leq |f(z)|\leq 2|f|_{B}\}\,.
\end{equation}
 Then
\begin{equation}
\bigg|\intav_{B}D|f|(y)\,dy\,\bigg|\leq C(n,\beta)\intav_{2B}|Df(z)|\chi_{E}(z)\,dz\,.
\end{equation}
\end{lemma}
\begin{proof}
First observe that by Lemma \ref{lemma1} it holds that
\begin{equation}\label{eq1}
\bigg|\intav_{B}D|f|(y)\,dy\,\bigg|\,\leq \frac{n}{r}\bigg[|f|_B-\intav_{\partial B}|f|\bigg]\,.
\end{equation}
Let then $|f(x)|=F(|x|)$, where $F:\re\setminus\{0\}\to[0,\infty)$, let $z$ denote the center point of $B$, $a:=|z|-r$, $b:=|z|+r$, and
\begin{equation}
A:=\{t\in2[a,b]\,:\,\frac{1}{2}|f|_{B}\leq F(t)\leq 2|f|_{B}\}.
\end{equation} 
Then we show that
\begin{equation}\label{eq3}
|f|_{B}-\intav_{\partial B}|f|\,\leq 2\int_{[a,b]}|F'(t)|\chi_{A}(t)\,dt\,.
\end{equation}
{The above inequality is more or less trivial: To prove it,  choose $t_0\in [a,b]$ such that $F(t_0)=|f|_B$ and choose 
$t_1\in [a,b]$ such that $|f|_{\partial B}=F(t_1)$. By \eqref{eq1} we have that $F(t_0)\geq F(t_1)$. In the case $F(t_1)\geq \frac{1}{2}|f|_B$ in $[a,b]$ 
 the claim follows by using the continuity, because in this case by the continuity we can assume without loss of generality that $[t_0,t_1]\subset A$ (or $[t_1,t_0]\subset A$).
 Otherwise, if 
 $F(t_1)< \frac{1}{2}|f|_B$, there exists $t_2\in [a,b]$ between $t_0$ and $t_1$ such that $F(t_2)=\frac{1}{2}|f|_B$, by the continuity of $F$ it clearly follows that
\begin{equation}
|f|_{B}-\intav_{\partial B}|f|\leq |f|_B\leq 2\int_{[t_0,t_2]}|F'(t)|\chi_{A}(t)\,dt\,\leq   2\int_{[a,b]}|F'(t)|\chi_{A}(t)\,dt\,.
\end{equation}}

Since $|Df(y)|\chi_{E}(y)=|F'(|y|)|\chi_A(|y|)$, Proposition \ref{radial1} yields that
\begin{equation}
\intav_{[a,b]}|F'(t)|\chi_{A}(t)\,dt\,\leq C(n)\intav_{B(z,2r)}|Df(y)|\chi_E(y)\,dy\,.
\end{equation}
Combining this with( \ref{eq1}) and (\ref{eq3}) implies the desired result. 
\end{proof} 

\begin{proposition}\label{prop1}
Suppose that $\beta\geq 0$, $f\in L^{1}_{loc}(\rn)$, and $B_1:=B(z_1,r_1)$ and $B_2:=B(z_2,r_2)$ are best balls for $M_{\beta}f$ such that $B_2\subset B(z_1,2r_1)$. Then it holds that
\begin{equation}
 |f|_{B_2}\geq \frac{1}{2^n}\bigg(\frac{r_1}{r_2}\bigg)^{\beta}|f|_{B_1}\,.
\end{equation}
\end{proposition}
\begin{proof}
Let $B:=B(z_1,2r_1)$. Since
$B_2$ is best ball and $B_2,B_1\subset B$, it holds that 
\begin{align*}
r_2^{\beta}|f|_{B_2}\geq (2r_1)^{\beta}|f|_{B}\geq (r_1)^{\beta}\frac{1}{2^n}|f|_{B_1}\,.
\end{align*}
This implies the claim.
\end{proof}

\section{Proof of the main Theorem}
Let  us fix $B_x:=B(z_x,r_x)\in \mathcal{B}_x$ for (almost) every $x\in E\,$, such that $r_x$ is the smallest possible (then by Lemma \ref{basic} item (3) we have $z_x=x+r_x\frac{DM_{\beta}f(x)}{|DM_{\beta}f(x)|}$), where
\begin{equation}
E:=\{\,x\,:\,DM_{\beta}f(x)\not=0\,\}\,.
\end{equation}
By the choise of radius we can see that $x\to r_x$ is an upper semicontinuous function then it is measurable function, thus $x\to z_x$ is also a measurable function. By Lemma \ref{basic}, it holds for almost all $x\in E$  that 
$B_x$ is of type
\begin{equation}
B_x= B(c_xx,|c_x-1||x|)\,,\,\,\text{ where }c_x\in\re\,.
\end{equation}
In the other words, this means that the center point of $B_x$ lies on the line containing $x$ and the origin, and $x$ lies on the boundary of $B_x$. 
For simplicity, let us yet denote the radius of $B_x$ by $r_x$, thus $r_x=|c_x-1||x|\,$. 
Observe first that for all $x\in E$ it holds that $c_x\geq 0$. To see this,  observe that otherwise
 (since $M_{\beta}f(x)=M_{\beta}f(-x)$) it follows that $B_x\in\mathcal{B}_{-x}$ and $-x\in B_x$, implying that $0=|DM_{\beta}f(-x)|=|DM_{\beta}f(x)|$, which is a contradiction.
We are going to use different type of estimates for $|DM_{\beta}f(x)|$ depending on how $B_x$ is located with respect to the origin. Indeed, let
\begin{align*}
E_1:&=\{x\in E\,:\,c_x>\frac{5}{4}\}\,,\,\,\,\,\,E_2:=\{x\in E\,:\,0\leq c_x<\frac{3}{4}\}\,,\,\,\text{and}\\
E_3:&=\{x\in E\,:\,\frac{3}{4}\leq c_x\leq \frac{5}{4}\}\,.
\end{align*}
Then we can estimate
\begin{align*}
&\int_{\rn}|DM_{\beta}f(x)|^qdx\,=\int_{E}|DM_{\beta}f(x)|^qdx=\int_{E}\bigg|r_x^{\beta}\intav_{B_x}Df(y)\,dy\,\bigg|^qdx\\
= &\int_{E}\frac{r_x^{q\beta}}{(\omega_n)^{q-1}r_x^{n(q-1)}}\bigg|\int_{B_x}Df(y)\,dy\,\bigg|^{q-1}\bigg|\intav_{B_x}Df(y)\,dy\,\bigg|\,dx\,\\
\leq &\,C(n,\beta)\norm{Df}_1^{q-1}\int_{E}\bigg|\intav_{B_x}Df(y)\,dy\,\bigg|\,dx\,\\
=&\,C(n,\beta)\norm{Df}_1^{q-1}\sum_{i=1}^3\int_{E_i}\bigg|\intav_{B_x}Df(y)\,dy\,\bigg|\,dx\,,
\end{align*}
where we used the fact $q\beta=n(q-1)\,$. Especially, the claim follows, if we can show that
\begin{equation}
\int_{E_i}\bigg|\intav_{B_x}Df(y)\,dy\,\bigg|\,dx\,\leq C(n,\beta) \norm{Df}_1\,,\text{ for }i=1,2,3\,.
\end{equation}

\textit{The case of $E_1$.} In this case the easiest type of estimate turns out to be sufficient. Indeed, 
\begin{align*}
\int_{E_1}\bigg|\intav_{B_x}Df(y)\,dy\,\bigg|\,dx\,\leq &\,\int_{E_1}\intav_{B_x}|Df(y)|\,dy\,dx\,\\
=&
\,\int_{\rn}|Df(y)|\int_{\rn}\frac{\chi_{B_x}(y)\chi_{E_1}(x)}{|B_x|}\,dx\,dy\,.
\end{align*}
For every $y\in\rn$ it holds that if $|x|\leq \frac{|y|}{2}$ and $y\in B_x$, then $r_x\geq|y|/4$. Moreover, if $\frac{|y|}{2}\leq |x|\leq |y|$, then $x\in E_1$ implies that
$r_x\geq \frac{1}{4}|x|\geq \frac{|y|}{8}\,$. Finally, if $|x|> |y|$ and  $x\in E_1$, then 
$B_x\subset \rn\setminus B(0,|y|)$, thus $y\not\in B_x$. By combining these, we conclude that for every $y\in\rn$
\begin{equation}
\int_{\rn}\frac{\chi_{B_x}(y)\chi_{E_1}(x)}{|B_x|}\,dx\,\leq\,C(n)\int_{B(0,|y|)}\frac{dx}{|B(0,|y|)|}\,=\,C(n)\,.
\end{equation}
\textit{The case of $E_2$.} In this case we recall the estimate from Lemma \ref{lemma3}, which yields that
\begin{equation}
\bigg|\intav_{B_x}Df(y)\,dy\,\bigg|\leq \,\intav_{B_x}|Df(y)|\frac{|y|}{|x|}\,\leq 4^n\intav_{B(0,|x|)}|Df(y)|\,\frac{|y|}{|x|}\,.
\end{equation}
Then the claim follows by 
\begin{align*}
&\int_{E_2}\bigg|\intav_{B_x}Df(y)\,dy\,\bigg|\,dx\,\leq \,4^{n}\int_{E_2}\intav_{B(0,|x|)}|Df(y)|\frac{|y|}{|x|}\,dy\,dx\\
=&4^{n}\int_{\rn}\,|Df(y)||y|\bigg(\int_{\rn}\frac{\chi_{B(0,|x|)}(y)\chi_{E_2}(x)}{\omega_n|x|^{n+1}}\,dx\,\bigg)\,dy\,\\
\leq &\frac{4^{n}}{\omega_n}\int_{\rn}\,|Df(y)||y|\bigg(\int_{\rn\setminus B(0, |y|)}\frac{dx}{|x|^{n+1}}\bigg)\,dy\,\\
= &C(n) \int_{\rn}|Df(y)|\,dy\,.
\end{align*}

\textit{The case of $E_3$.} 
In this case we will exploit the estimate from Lemma \ref{lemma2}. For this, let us denote
for every $x\in E_3$ that
\begin{equation}
A_x:=\{y\in 2B_x\,:\,\frac{1}{2}|f|_{B_x}\leq |f(y)|\leq 2|f|_{B_x}\}\,.
\end{equation}

Since $x\in E_3$ implies that $r_x\leq\frac{|x|}{4}$, Lemma \ref{lemma2} yields that for every $x\in E_3$ it holds that
\begin{equation}
\bigg|\intav_{B_x}Df(y)\,dy\,\bigg|\,\leq\,C(n,\beta)\intav_{2B_x}|Df(y)|\chi_{A_x}(y)\,dy\,.
\end{equation}
Therefore, 
\begin{align*}
&\int_{E_3}\bigg|\intav_{B_x}Df(y)\,dy\,\bigg|\,dx\,\leq C\,\int_{E_3}\intav_{2B_x}|Df(y)|\chi_{A_x}(y)\,dy\,dx\\
=&C\int_{\rn}\,|Df(y)|\bigg(\int_{\rn}\frac{\chi_{2B_x}(y)\chi_{A_x}(y)\chi_{E_3}(x)}{|2B_x|}\,dx\,\bigg)\,dy\,.
\end{align*}
Consider above the inner integral for fixed $y\in\rn$. Firstly, suppose that 
\begin{equation}
\chi_{2B_{x_0}}(y)\chi_{A_{x_0}}(y)\not =0\,\,\text{ and }\,\chi_{2B_{x_1}}(y)\chi_{A_{x_1}}(y)\not =0\,\,,\text{ for some }\,x_0,x_1\in\rn\,. 
\end{equation}
Observe that if this kind of points does not exist, the desired estimates are trivially true.
By the definition, the above means that
\begin{align*}
&\frac{1}{2}|f|_{B_{x_0}}\leq |f(y)|\leq 2|f|_{B_{x_0}}\,\,,\,\text{  and }\\
&\frac{1}{2}|f|_{B_{x_1}}\leq |f(y)|\leq 2|f|_{B_{x_1}}\,.
\end{align*}
Especially, it follows that
\begin{equation}
\frac{1}{4}|f|_{B_{x_0}}\leq |f|_{B_{x_1}}\leq 4|f|_{B_{x_0}}\,.
\end{equation}
Let  $r_0:=rad(B_{x_0})$ and $r_1:=rad(B_{x_1})$ and assume that $r_1\geq r_0$. Since $y\in 2B_{x_0}\cap 2B_{x_1}$, it follows that $B_{x_0}\subset 8B_{x_1}$. By Proposition \ref{prop1}, it follows that
\begin{equation}
|f|_{B_{x_0}}\geq \frac{1}{8^n}\bigg(\frac{r_1}{r_0}\bigg)^{\beta}|f|_{B_{x_1}}\geq   \frac{1}{8^n}\bigg(\frac{r_1}{r_0}\bigg)^{\beta}\frac{1}{4}|f|_{B_{x_0}}\,,
\end{equation}
implying that
$r_1\leq 8^{\frac{n+1}{\beta}}r_0\,$.
If $r_1\leq r_0$, symmetric argument gives that $r_0\leq 8^{\frac{n+1}{\beta}}r_1\,$. Summing up, it follows that
\begin{equation}
8^{-\frac{n+1}{\beta}}\leq \frac{rad(B_{x_0})}{rad(B_{x_1})}\leq 8^{\frac{n+1}{\beta}}\,. 
\end{equation}
Indeed, this means that if $\chi_{2B_{x}}(y)\chi_{A_{x}}(y)\not =0$ , then 
\begin{equation}\label{eqc}
|x-y|\leq C(n,\beta)rad(B_{x_0})\, \text{ and }\,|B_x|\geq C(n,\beta) |B_{x_0}|\,.
\end{equation}
Naturally, (\ref{eqc}) holds also if $x_0$ is replaced by $x_1$. Finally, this implies that 
\begin{align*}
\int_{\rn}\frac{\chi_{2B_x}(y)\chi_{A_x}(y)\chi_{E_3}(x)}{|2B_x|}\,dx\,&\leq C(n,\beta)\int_{B(y,C(n,\beta)rad(B_{x_0}))}\frac{dx}{|B_{x_0}|}\\
\,&\leq\,\tilde{C}(n,\beta)\,.
\end{align*} 
Since this holds for all $y\in\rn$, the proof is complete.

\section{Acknowledgments}
\noindent H.L acknowledges M. Parviainen, J. Kinnunen and the Academy of Finland 
for the financial support. J.M. acknowledges J. Kinnunen, Aalto University and Academy of Finland for the support. The authors are thankful to Juha Kinnunen for helpful discussions and guidance during the preparation of this manuscript. The authors thank Emanuel Carneiro for suggesting to think about this problem. The authors also acknowledge the referee for the valuable comments and suggestions.

\end{document}